\documentclass[12pt]{amsart}
\usepackage[margin=1in]{geometry}
\usepackage{graphicx}
\usepackage{amssymb}
\usepackage{amsmath}
\usepackage{amsfonts}
\usepackage{amstext}
\usepackage{amsthm} 
\usepackage{setspace}
\usepackage{array}
\usepackage{textcomp}
\usepackage{tipa}
\usepackage{graphicx}

\usepackage{wrapfig}
\usepackage{caption}
\usepackage{subcaption}

\newcommand{\be}{\begin{equation} }
\newcommand{\ee}{\end{equation} }
\newcommand{\bee}{\begin{eqnarray} }
\newcommand{\eee}{\end{eqnarray} }

\numberwithin{equation}{section}

\newtheorem{Lem}{Lemma}
\newtheorem*{thm}{Theorem}
\newtheorem{Thm}{Theorem}

\newtheorem{definition}{Definition}
\newtheorem{remark}{Remark}

\newcommand{\set}[1]{\ensuremath{\{#1\}}}

\newcommand{\E}[1]{\ensuremath{\mathbb E\left[ #1 \right]}}
\newcommand{\EE}[1]{\mathbb E}

\newcommand{\ep}{\ensuremath{\epsilon}}

\newcommand{\C}{\ensuremath{\mathbb C}}

\newcommand{\W}{\ensuremath{\Omega}}

\newcommand{\twiddle}[1]{\ensuremath{\widetilde{#1}}}

\newcommand{\setst}[2]{\ensuremath{\left\{#1\,\middle|\,#2\right\}}}

\newcommand{\abs}[1]{\left\lvert #1 \right\rvert}

\newcommand{\lr}[1]{\ensuremath{\left(#1\right)}}

\newcommand{\dell}{\ensuremath{\partial}}

\newcommand{\und}[1]{\underline{#1}}

\begin{document}
\author{Boris Hanin}

\address{Department of Mathematics, MIT, Cambridge, MA 02142, USA}
\email[B. Hanin]{bhanin@mit.edu}

\thanks{Research partially supported by NSF grant DMS-1400822.}

\title[Pairing of Zeros and Critical Points]{Pairing of Zeros and Critical Points for Random Polynomials}
\setcounter{section}{-1}
\maketitle
\begin{abstract}
Let $p_N$ be a random degree $N$ polynomial in one complex variable whose zeros are chosen independently from a fixed probability measure $\mu$ on the Riemann sphere $S^2.$ This article proves that if we condition $p_N$ to have a zero at some fixed point $\xi\in S^2,$ then, with high probability, there will be a critical point $w_\xi$ a distance $N^{-1}$ away from $\xi.$ This $N^{-1}$ distance is much smaller than the $N^{-1/2}$ typical spacing between nearest neighbors for $N$ i.i.d. points on $S^2.$ Moreover, with the same high probability, the argument of $w_\xi$ relative to $\xi$ is a deterministic function of $\mu$ plus fluctuations on the order of $N^{-1}.$
\end{abstract}

\section{Introduction}
This article concerns a surprising relationship between zeros and critical points of a random polynomial in one complex variable. To introduce our results, consider Figure 1, which shows the zeros and critical points of $p(z)=z^9-1$ and of $q(z)=p(z)(z-\xi)$ for various $\xi.$ While the zeros and critical points of $p$ are quite far apart, most zeros of $q$ seem to have a unique nearby critical point. This effect becomes more pronounced for a polynomial whose zeros are chosen at random as in Figure 2. What accounts for such a pairing? How close is a zero to its paired critical point? Why does the pairing break down in some places? Why is there such a rigid angular dependence between a zero and its paired critical point? We give in \S \ref{S:Electrostatics} intuitive answers to these questions using an interpretation of zeros and critical points that relies on electrostatics on the Reimann sphere $S^2.$ This physical heuristic, in turn, guides the proofs of our main results, Theorems \ref{T:SCor} and \ref{T:Main}.

\begin{figure}
 \label{fig:1}
   \centering
     \includegraphics[width=0.32\textwidth]{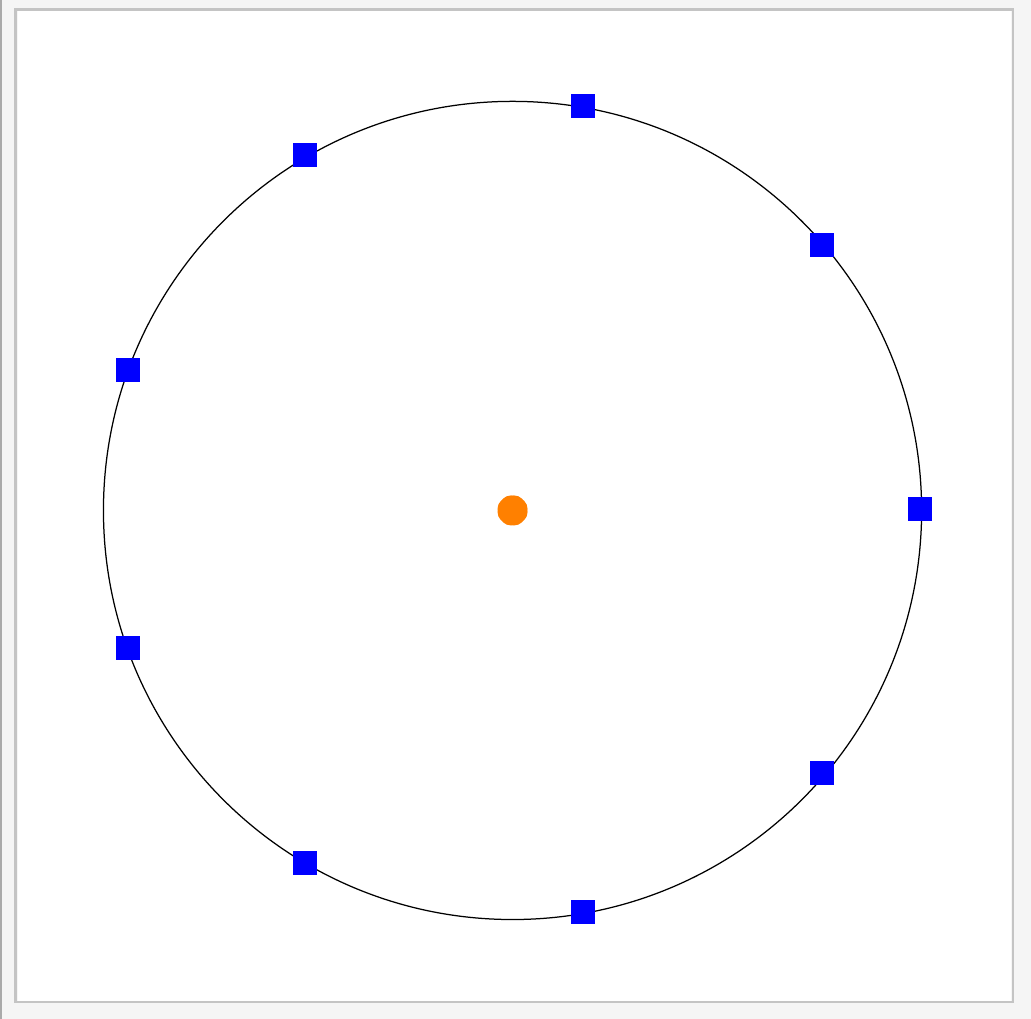}
     \includegraphics[width=0.32\textwidth]{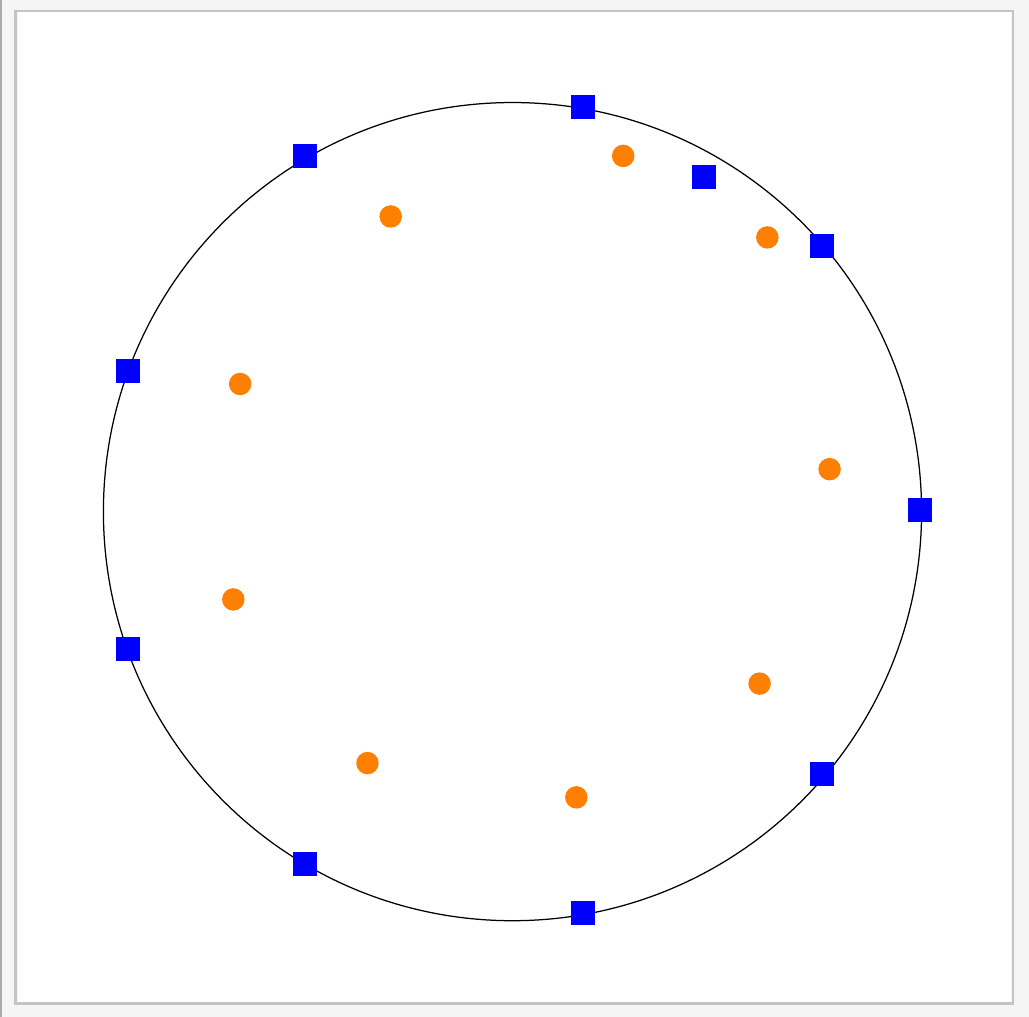}
  \includegraphics[width=0.32\textwidth]{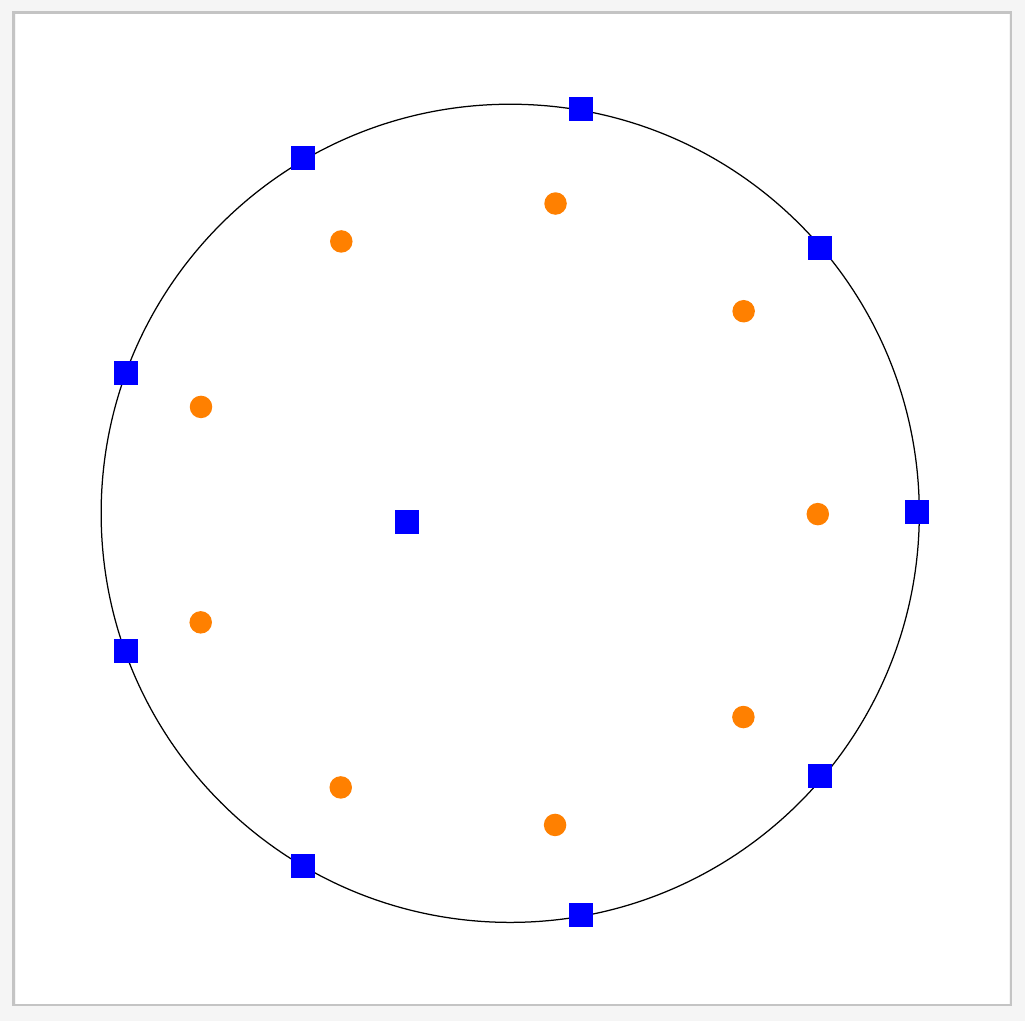}
    \caption{Figures from left to right display zeros (blue squares) and critical points (orange disks) of $p(z)=z^9-1$ and of $q(z)=p(z)(z-\xi)$ for various $\xi.$}
\end{figure}

\begin{figure}
 \label{fig:2}
   \centering
     \includegraphics[width=0.32\textwidth]{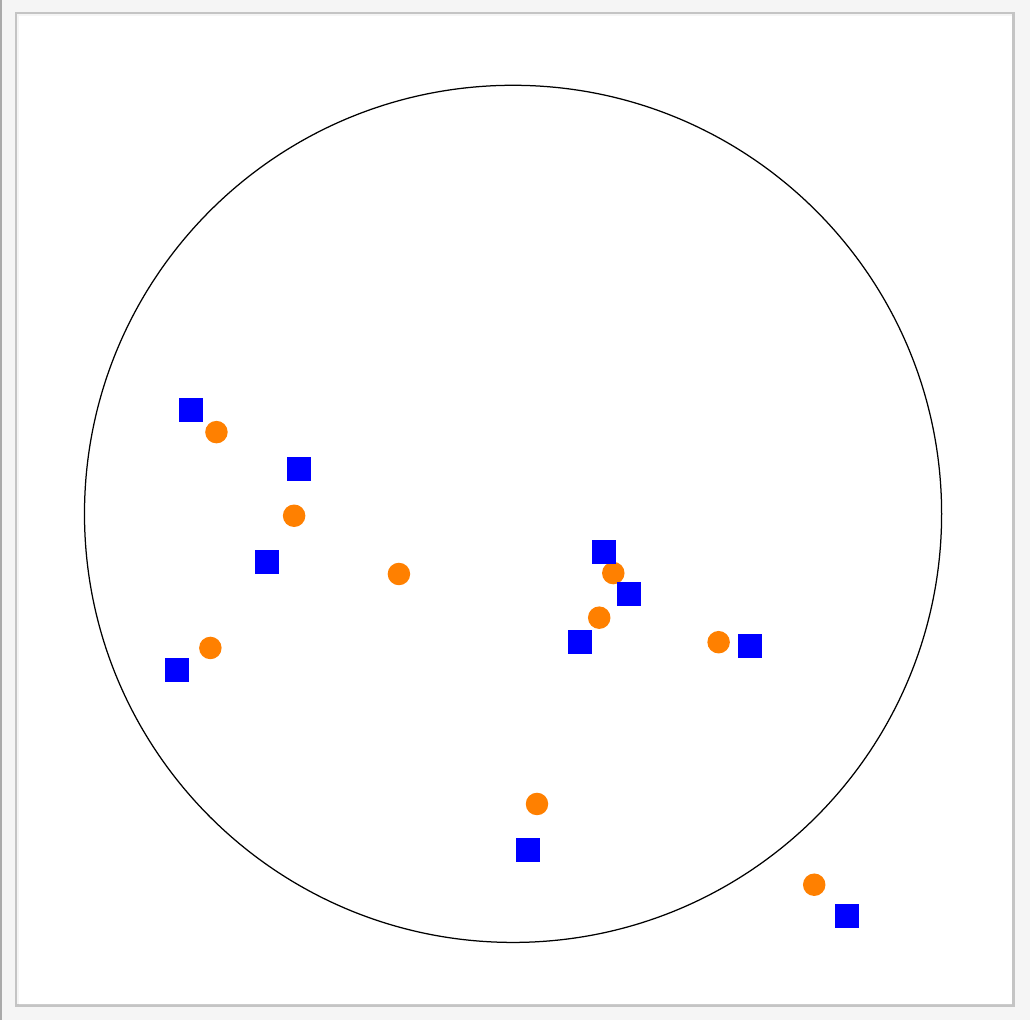}
     \includegraphics[width=0.32\textwidth]{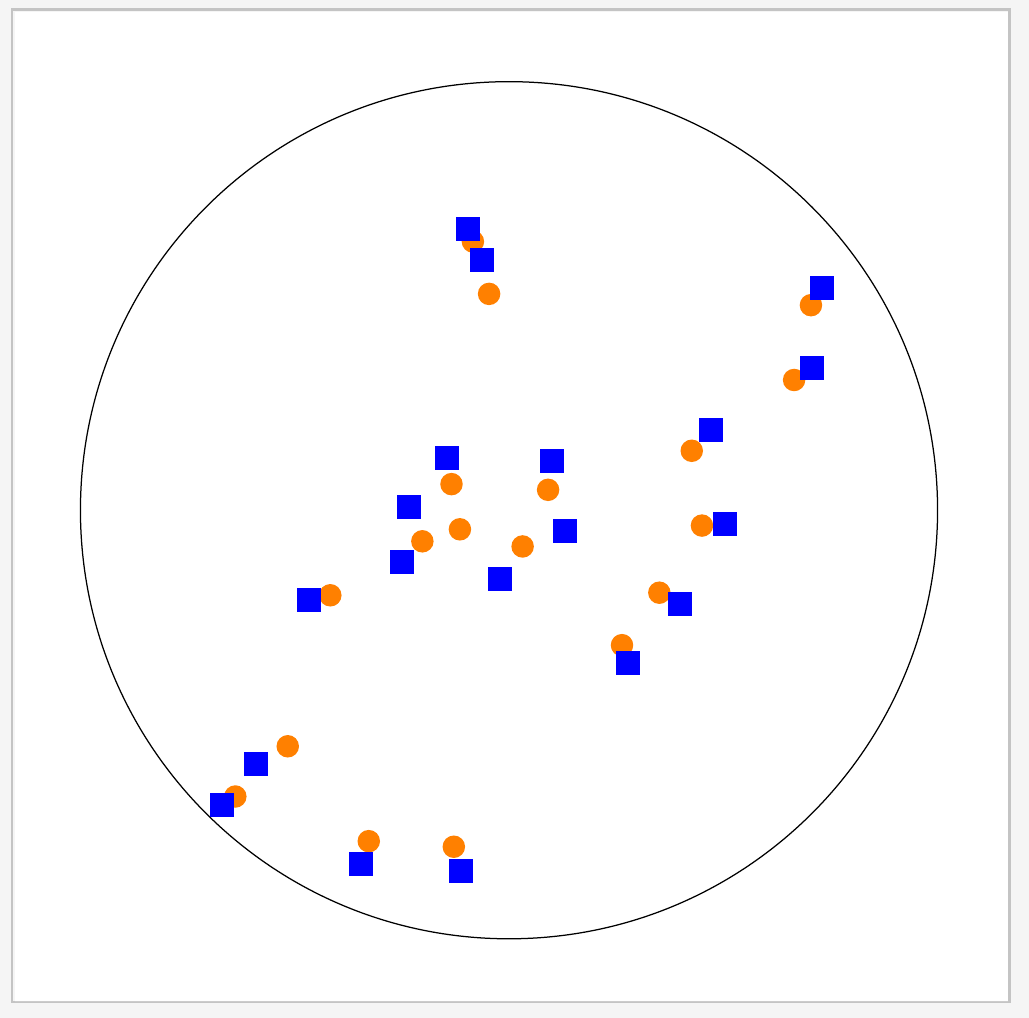}
  \includegraphics[width=0.32\textwidth]{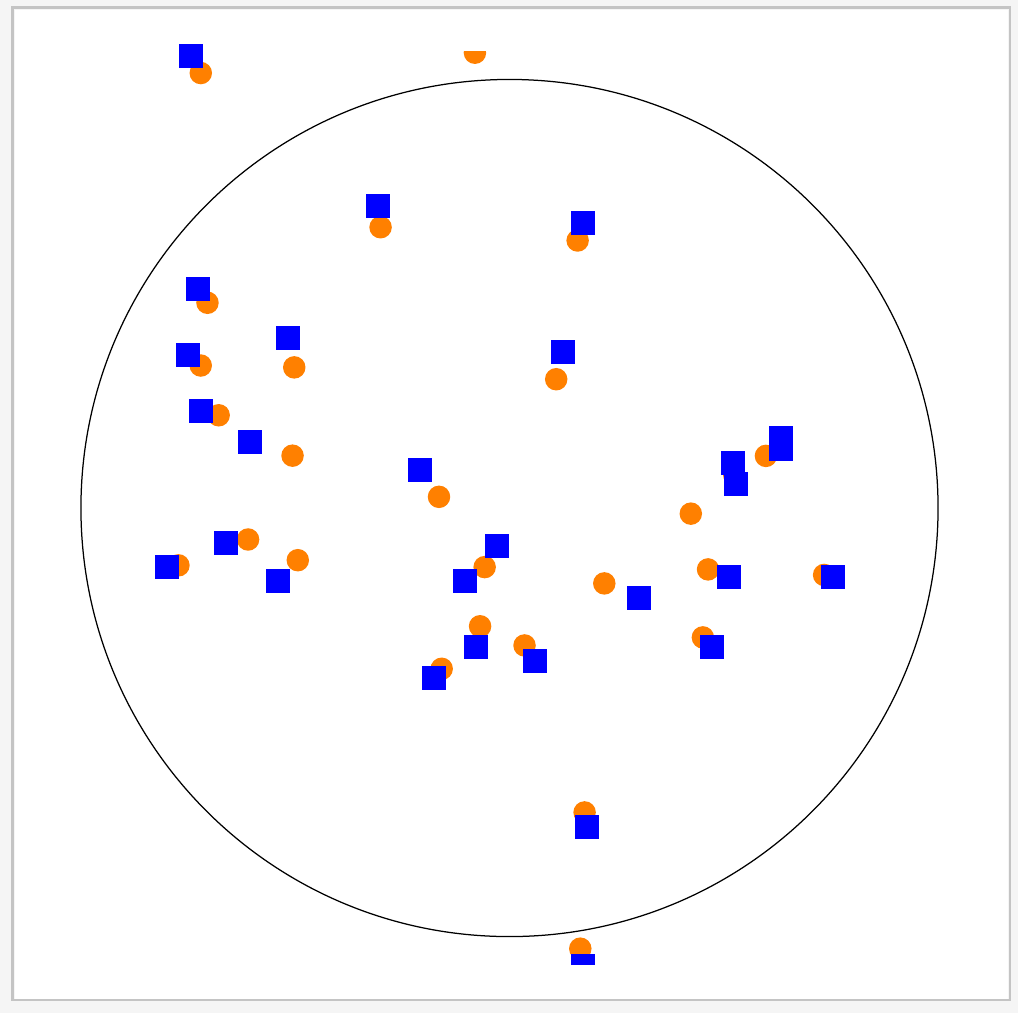}
    \caption{Figures from left to right display zeros (blue squares) and critical points (orange disks) of a polynomial with $10,20,30$ zeros each chosen independently from the standard Gaussian measure on $\C.$}
\end{figure}

There is a vast literature on the distribution of zeros of random polynomials, and we will not attempt to survey it here. We simply mention that they have been studied from the point of view of random analytic functions (cf \cite{S}); random matrix theory, where eigenvalues are zeros of characteristic polynomials of random matrices (cf \cite{AGZ}); and determinantal point processes (cf \cite{HKPV}). Previous results specifically relating zeros and critical points of random polynomials are much more limited, however. This is somewhat surprising because there are many interesting deterministic theorems that restrict the possible locations of critical points of a polynomial in terms of the locations of its zeros. We recall two such results, and refer the reader to Marden's book \cite{M} for many more.
\begin{thm}[Gauss-Lucas]
The critical points of a polynomial in one complex variable lie inside the convex hull of its zeros.
\end{thm}

\begin{thm}[Theorem 3.55 in \cite{Tit}]
Let $f$ be a non-zero holomorphic function on a simply connected domain $\W\subseteq \C,$ and take $\Gamma$ to be a smooth closed connected component of the level set $\set{\abs{f(z)}=t}$ for some $t.$ Write $U=\set{\abs{f(z)}< t}$ for the open domain bounded by $\Gamma.$ Then 
\[\#\set{ f(z)=0}\cap U= \lr{\#\set{\frac{d}{dw}f(w)=0}\cap U}+1.\] 
\end{thm}

We are aware of only three previous works concerning the sort of a pairing between zeros and critical points discussed here. From the math literature, there are the author's two articles \cite{H1,H2} which study a large class of Gaussian random polynomials called Hermitian Gaussian Ensembles (HGEs). The simplest HGE is the $SU(2)$ or Kostlan ensemble:
\begin{equation}
p_N^{SU(2)}(z):=\sum_{j=0}^N a_j ~\sqrt{\binom{N}{j}}~z^{\, j},\qquad a_j\sim N(0,1)_{\C}\,\,\, i.i.d.\label{E:SU2}
\end{equation}
Even for the $SU(2)$ ensemble, the pairing of zeros and critical points was a new result (see Figure 2). The proofs in \cite{H1,H2} are not elementary, however, because the distribution of zeros and critical points for HGEs is highly non-trivial and is written in terms of so-called Bergman kernels. Moreover, most of the theorems in \cite{H1,H2} do not discuss the angular dependence between a zero and its paired critical point. The present article, in constrast, studies the simplest possible ensembles of random polynomials from the point of view of the joint distribution of the zeros. Namely, we fix some some number of zeros and choose the others uniformly and independently from a fixed probability measure on $S^2.$ In this situation, we are able to give the first completely elementary proof of the pairing of zeros and critical points for a random polynomial. 

The other article we are aware of is the heuristic work of Dennis-Hannay \cite{DH} from the physics literature. They give an electrostatic explanation for why, for certain special kinds of random polynomials, zeros with a large modulus should be paired to a critical point. In \S \ref{S:Electrostatics} we give a somewhat different and more flexible electrostatic argument that explains the pairing of zeros and critical points. Our reasoning also predicts the distance between from a zero to its paired critical point as well as the existence of regions where such a pairing breaks down. 

To conclude, let us mention the works of Kabluchko \cite{Kab}, Pemantle-Rivlin \cite{PR}, and Subramanian \cite{Sub}, which study the empirical measure of critical points for ensembles of random polynomials similar to the ones we consider here. We also point the reader to the work of Nazarov-Sodin-Volberg \cite{NSV} and the recent article of Feng \cite{Fen}, which both concern the critical points of random holomorphic functions with respect to a smooth connection. 

 \begin{figure}[h]
\vskip -.2cm
\centering
\includegraphics[width=.48\textwidth]{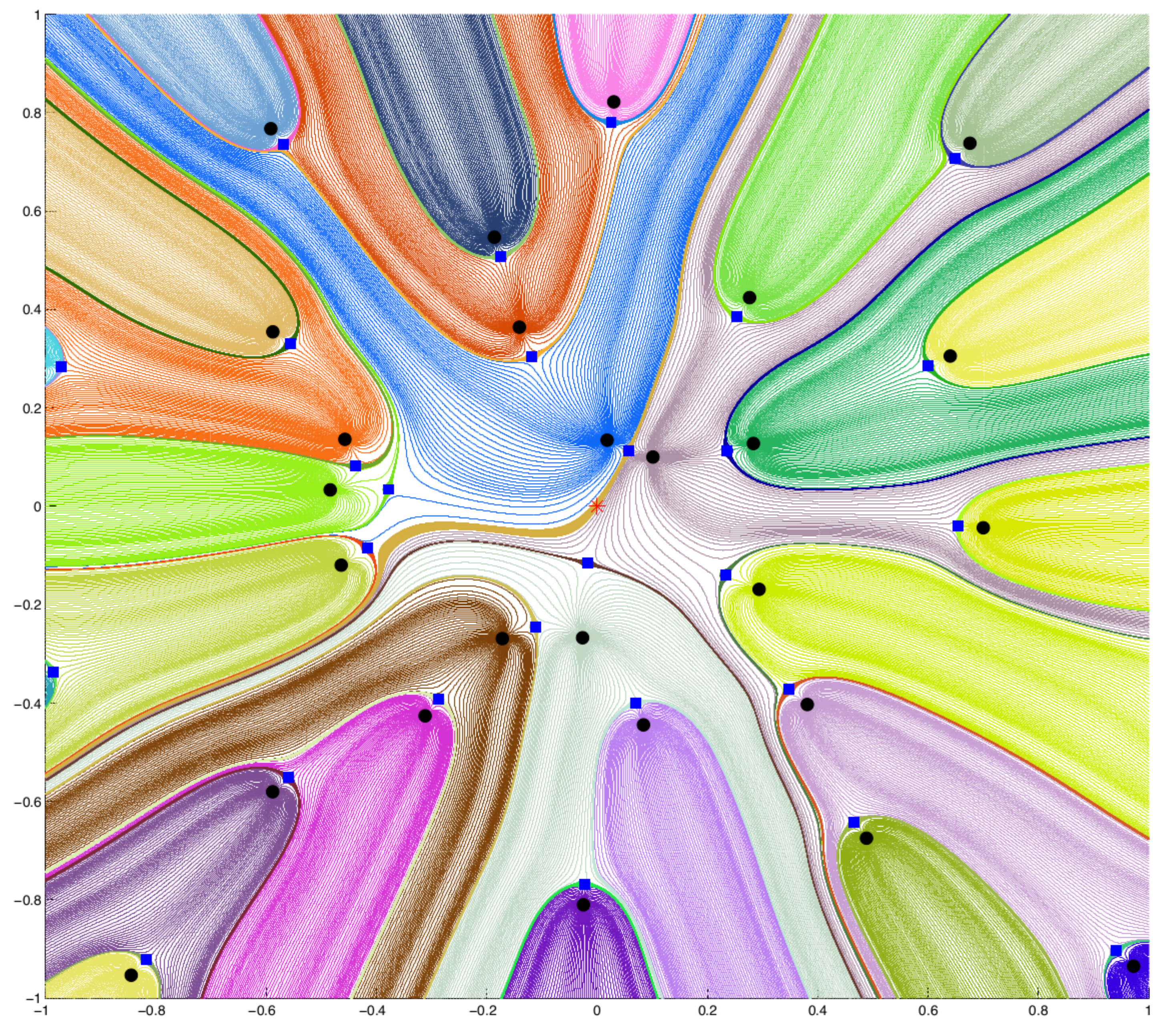}
\includegraphics[width=.48\textwidth]{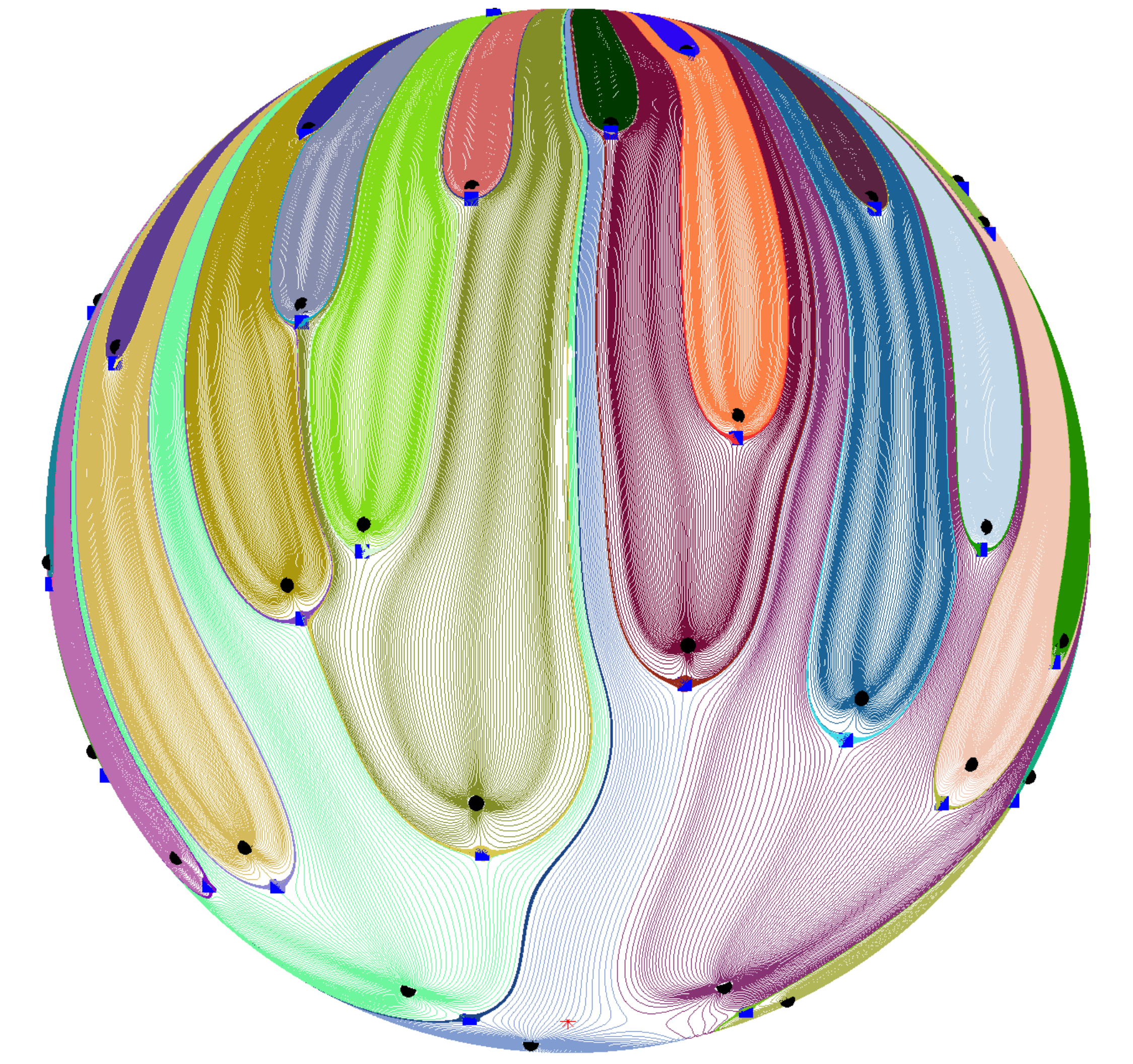}
\label{F:Unscaled SU2 50}
\caption{Zeros (black disks) and critical points (blue squares) for a degree $50\,\, SU(2)$ polynomial $p$ (defined in \eqref{E:SU2}) displayed on the left in coordinates $\C \cong S^2\backslash \set{\infty}$ and on the right directly on $S^2$. The colored lines are gradient flow lines for $-\abs{p(z)}^2$. Lines of the same color terminate at the same point. The origin is denoted by a red asterisk in both figures.}
\end{figure}

\section{Acknowledgements}
I am grateful to Leonid Hanin and Steve Zelditch for many useful comments on earlier drafts of this article. I am also indebted to Ron Peled who shared with me the Matlab code (originally written by Manjunath Krishnapur) that I modified to create Figures 3 and 4. Finally, I would like to acknowledge Bruce Torrence and Paul Abbott, whose Mathematica demonstration \cite{AT} I modified to create Figures 1 and 2. 

\section{Electrostatic Interepretation of Zeros and Critical Points}\label{S:Electrostatics}
The idea that zeros and critical points of a complex polynomials have an electrostatic interpretation goes back to Gauss \cite[Preface and \S 2]{M}. And the observation that zeros with a large modulus should be paired to a critical point in certain special kinds of random polynomials (and for certain random entire functions) was stated by Dennis and Hanny in \cite{DH}. They give a heuristic explanation for this pairing, which is similar in spirit to ours, but does not use that polynomials have a high multiplicity pole at infinity.

We begin by explaining Gauss's proof of the Gauss-Lucas Theorem. Suppose $p_N$ is a degree $N$ polynomial and $\xi_j,\, j=0,\ldots, N-1$ are its zeros. The critical points of $p_N$ are solutions to
\[\frac{d}{dw}p_N(w)=0 \qquad \longleftrightarrow \qquad \dell \log\abs{p_N(w)}^2=\sum_{j=0}^{N-1}\frac{1}{w-\xi_j}=0.\]
A basic observation is that $\lr{\overline{z-\xi_j}}^{-1}$ is the electric field at $z$ from a $+1$ charge at $\xi_j.$ The sum $\dell\log\abs{p_N(z)}^2$ is therefore the complex conjugate of the total electric field $E_N(z)$ at $z$ from positive point charges placed at each $\xi_j.$ This means 
\[\frac{d}{dw}p_N(w)=0 \qquad \longleftrightarrow \qquad \text{electric field satisfies }E_N(w)=0.\]
Since all charges have the same sign, $E_N$ cannot vanish outside of their convex hull. Hence the Gauss-Lucas Theorem.

The preceeding argument relied very much on the particular coordinates on $S^2$ (the convex hull is not a coordinate-free notion). However, the electrostatic interpretation of critical points can itself be done in a coordinate invariant way. We start by viewing $p_N$ as a meromorphic function on $S^2.$ Then
\[\Delta\log \abs{p_N}^2=Div(p_N)=-N\delta_{\infty}+\sum_{p_N(\xi)=0}\delta_{\xi},\]
where we've written $\Delta =\frac{i}{2\pi}\dell \overline{\dell}$ for the Laplacian, and the equality is in the sense of distributions. This means that 
\[E_N(w)=\dell\log \abs{p_N(w)}^2\] 
is a one-form that at any $w\in S^2$ gives the (complex conjugate of the) electric field at $w$ from charge distributed according to $Div(p_N).$ Gauss's choice of coordinates makes $\infty$ infinitely far way and and so the contribution to $E_N$ from the $-N$ charges at infinity was zero. To get a true electric field, we need to covert $E_N$ to a vector field using a metric. However, the points where $E_N$ vanishes (i.e. the critical point of $p_N$) are independent of such a choice.

This point of view is closely related to the classical notion of a polar derivative of a complex polynomial \cite[\S 3]{M}, which can be used to prove some coordinate free versions of the Gauss-Lucas Theorem such as Laguerre's Theorem \cite[p. 49]{M}. That it should imply a pairing of zeros and critical points never seems to have been observed, however.

It is precisely the charge of size $N$ at $\infty$ that clarifies why zeros and critical points come in pairs. To see this, let us consider the case of a degree $N$ polynomial $p_{N,\xi}$ that has a zero at a fixed point $\xi\in S^2,$ while its remaining zeros $\xi_1,\ldots, \xi_{N-1}$ are chosen independently from the uniform measure $\mu$ on $S^2$. We must explain why, with high probability, there is a point $w_\xi$ near $\xi$ at which the electric field $E_N(w_\xi)$ vanishes. In the holomorphic coordinate $w$ centered at $\infty,$ we have
\begin{equation}
p_{N,\xi}(w)=\frac{1}{w^N}\lr{w-\xi}\prod_{j=1}^{N-1}\lr{w-\xi_j}\qquad \xi_j\sim \mu\quad i.i.d.\label{E:patinfty}
\end{equation}
Thus,
\begin{equation}
E_N(w)=-\frac{N}{w}+\frac{1}{w-\xi}+\sum_{j=1}^{N-1}\frac{1}{w-\xi_j},\label{E:avgE}
\end{equation}
The first term in \eqref{E:avgE} is the contribution from the $-N$ charges at $\infty,$ while the second comes from the $+1$ charge at $\xi$, which is also of order $N$ if $\abs{w-\xi}\approx N^{-1}.$ The third term is a sum of iid random variables. It is equal to zero on average (cf \eqref{E:avgunif}). Hence, heuristically, it is should be on the order of $N^{1/2}$ by the central limit theorem. Therefore, to leading order in $N,$ the electric field near $\xi$ is very close to its average
\begin{equation}
E_N(w)\approx \E{E_N(w)}=-\frac{N}{w}+\frac{1}{w-\xi}.\label{E:avgE2}
\end{equation}
As long as $\xi\not \in \set{\underline{0},\infty}$ (here $\underline{0}$ is the antipodal point on $S^2$ to $\infty$) there will be a unique solution 
\begin{equation}
w_{N,\xi}=\xi\lr{1-\frac{1}{N}}^{-1}\label{E:critatinfty}
\end{equation}
to 
\[\E{E_N(w)}=0\] 
in the regime where the approximation \eqref{E:avgE2} is valid. Note that the distance from $w$ to $w_\xi$ is on the order of $N^{-1}.$ The true critical point of $p_N$ will, by Rouch\'e's Theorem, therefore be a small perturbation of $w_{N,\xi}.$ Note that $w_\xi$ is a bit farther away from $\infty$ than $\xi$ and hence is closer to $\underline{0}$ as shown in Figures 1-4. The condition that $\xi\not \in \set{\underline{0},\infty}$ is not an artifact of our reasoning. Figures 1-4 clearly show the existence of regions where the pairing of critical points breaks down: at $\xi=\underline{0},$ the contribution to $E_N$ from the charges at $\infty$ precisely cancels by symmetry. Near $\underline{0}$ the field $E_N$ is controlled by the nearby zeros whose statistical fluctuations cannot be ignored.

Now let us suppose that the zeros $\xi_1,\ldots, \xi_{N-1}$ are still uniformly distributed but now according to an arbitrary probability measure $\mu$ on $S^2.$ There will still be a pairing of zeros and critical points, but the pairs will no longer necessarily align with $\infty.$ Indeed, the main difference in this case is that instead of \eqref{E:avgE2}, the electric field near $\xi$ will have a non-zero contribution from the average of the third term in \eqref{E:avgE}. To leading order in $N,$ we have
\begin{equation}
E_N(w)\approx \E{E_N(w)}\approx N\lr{-\frac{1}{w}+\phi_\mu(w)}+\frac{1}{w-\xi},\label{E:avgE3}
\end{equation}
where
\begin{equation}
\phi_\mu(w)=\int_{\C}\frac{d\mu(\zeta)}{w-\zeta}\label{E:phidef}
\end{equation}
is the (complex conjugate of the) average electric field at $w$ from a zero distributed according to $\mu.$ To ensure a unique point $w_{N,\xi}$ where the right hand side of \eqref{E:avgE3} vanishes that is near $\xi$ we must ask that $\xi \not \in S_\mu,$ where
\begin{equation}
S_\mu:=\setst{w\in \C}{-\frac{1}{w}+\phi_\mu(w)\in \set{0,\infty}}.\label{E:Sdef}
\end{equation}
Expilcitly,
\begin{equation}
w_{N,\xi}= \xi\lr{1-\frac{1}{N}\cdot \frac{1}{\phi_\mu(\xi)\cdot \xi-1}}\label{E:gencrit}
\end{equation}
plus an $O(N^{-2})$ error and 
\[\arg(w_{N,\xi}-\xi)=\arg(\xi)-\arg\lr{\frac{1}{\xi}-\phi_\mu(\xi)}+O(N^{-1})\]
Before stating the rigorous results of this article, we remark that the heuristic argument given here does not make strong use of the iid nature of the zeros of $p_N$. It does crucially rely on the assumption that they are well-spaced and not too correlated, however. 

\section{Main Result}
Our main result, Theorem \ref{T:SCor}, can be stated loosely as follows. Consider a degree $N$ polynomial $p_N,$ viewed as a meromorphic function on $S^2.$ Suppose $p_N$ has a zero at a fixed point $\xi,$ while its other zeros are randomly and independently selected. Then, we are likely to observe a critical point $w_\xi$ of $p_N$ a distance about $N^{-1}$ away from $\xi$. 

Note that if we choose $N$ independent points at random from the uniform measure on $S^2,$ then the typical spacing between nearest neighbors is on the order of $N^{-1/2},$ which is much larger than the $N^{-1}$ spacing between a zero and its paired critical point. Observe also that Theorem \ref{T:SCor} is genuinely probabilistic and does not hold for polynomials of the form $z^N-R^N$. Nonetheless, as explained in the Introduction, multiplying $z^9-1$ by a single linear factor already makes the zeros and critical points of the resulting polynomial come in pairs (see Figure 1).

Let us write $\mathcal P_N$ for the space of polynomials of degree at most $N$ in one complex variable. Since the zeros and critical points of $p_N\in \mathcal P_N$ are unchanged after multiplication by a non-zero constant, we study random the zeros and critical points of a random polynomial by putting a probability measure directly on the projectivization $\mathbb P \lr{\mathcal P_N}$ as follows. 
\begin{definition}\label{D:Pdef}
Fix $\xi \in S^2$ and a probability measure $\mu$ with a bounded density with respect to the uniform measure on $S^2.$ Define $[p_{N,\xi}]$ to be a random element in $\mathbb P \lr{\mathcal P_N}$ with a (deterministic) zero at $\xi$ and $N-1$ (random) zeros $\set{\xi_1,\ldots, \xi_{N-1}}$ distributed according to the product measure $\mu^{\otimes (N-1)}$ on $\lr{S^2}^{N-1}.$  
\end{definition}
Slightly abusing notatoin, will hencefore write $p_{N,\xi}$ for any representative of $[p_{N,\xi}].$ We also identify once and for all polynomials with meromorphic functions on $S^2$ that have a pole at the distinguished point $\infty \in S^2,$ and we will write $\underline{0}$ for the antipodal point to $\infty.$ With $w$ denoting the usual holomorphic coordinate on $S^2$ centered at $\infty,$ $p_{N,\xi}$ is given by \eqref{E:patinfty}.

Let $S_\mu$ be defined as in \eqref{E:Sdef}. For each $\xi\in S^2\,\backslash\, S_\mu,$ we define $w_{\xi,N}$ to be the unique solution to the averaged critical point equation
\begin{equation}
\E{\dell_w \log \abs{p_{N,\xi}(w)}^2}=0 \label{E:critdef}
\end{equation}
whose distance from $\xi$ is on the order of $N^{-1}.$ The point $w_{N,\xi}$ is therefore a point whose distance from $\xi$ is approximately $N^{-1}$ where the electric field is expected to vanish. See \eqref{E:gencrit} for an asymptotic formula for $w_{\xi,N}$.
\begin{Thm}[Pairing of Single Zero and Critical Point]\label{T:SCor}
Let $\mu$ be any probability measure on $S^2$ that has a bounded density with respect to the uniform measure, and fix $\xi \in S^2\backslash S_\mu,$ where $S_\mu$ is defined in \eqref{E:Sdef}. Consider $p_{N,\xi}$ as in Definition \ref{D:Pdef}. Fix $r>0,$ and write $\Gamma_N$ for the geodesic circle of radius $r N^{-1}$ centered at $w_{N,\xi}$. Suppose that $\xi \not \in \Gamma_N$ for all $N.$ Then, for any $\delta\in (0,1),$ there exists $C=C(r, \delta)>0$ so that for all $N$
\[\mathbb P\lr{\exists ! ~w\text{ inside }\Gamma_N~~\text{ s.t. }~~ \frac{d}{dw}p_{N,\xi}(w)=0}\geq 1- C \cdot N^{-\delta}.\]
\end{Thm}
\begin{remark}\label{R:GenCont}
The conclusion of Theorem \ref{T:SCor} is actually true for any simple closed contour $\Gamma_{N,\xi}$ with winding number $1$ around $w_{\xi,N}$ that does not pass through $\xi$ and satisfies:
\begin{enumerate}
\item[(i)] There exists $c_1>0$ so that 
\[\inf_{w\in \Gamma_{N,\xi}}\big|\,\E{\dell \log\abs{p_{N,\xi}(w)}^2}\,\big|\geq c_1\cdot N.\]
\item[(ii)] There exists $c_2>0$ so that for all $N$
\[\sup_{w\in \Gamma_{N,\xi}} d_{S^2}(w,\xi)\leq c_2\cdot N^{-1},\]
where $d_{S^2}$ is the usual distance function on $S^2.$
\end{enumerate}
\end{remark}

\subsection{Generalizations of Theorem \ref{T:SCor}}
Theorem \ref{T:SCor} can generalized in many different ways. A particularly simple extension concerns the simultaneous pairing of $N^\alpha$ zeros and critical points for any $\alpha \in [0,1).$ To give a exact statement, consider for each $N$ a finite collection of at most $N$ points $\Xi_N\subseteq S^2\backslash S_\mu.$ Define $p_{N,\Xi_N}$ to be a random degree $N$ polynomial that vanishes at each $\xi\in \Xi_N$ and whose other zeros are chosen independently from $\mu.$ As above, this defintion actually specifies an equivalence class $[p_{N,\Xi_N}]$ in $\mathbb P\lr{\mathcal P_N}$ all of whose representatives have the same zeros and critical points. 

\begin{Thm}[Pairing of $N^\alpha$ Zeros and Critical Points]\label{T:Main}
Fix $\alpha\in [0,1)$ and $\delta\in (0,1-\alpha).$ For every $N\geq 1,$ choose an integer $n(N)$ between $1$ and $N^{\alpha}$ and a collection of $n(N)$ points $\Xi_N\subseteq S^2\backslash S_\mu$ satisfying
\begin{enumerate}
\item[(A)] There exists $\ep\in (0,1)$ so that for every $N$ and every pair of distinct point $\xi_1,\xi_2\in \Xi_N$ we have $\abs{\xi_1-\xi_2}>N^{-1/2+\ep/2}.$
\end{enumerate}
Define $p_{N,\Xi_N}$ as above, and for each $N$ and every $\xi\in \Xi_N,$ let $\Gamma_{N,\xi}$ be the geodesic circle of radius $rN^{-1}$ centered at $w_{N,\xi}$. If $\xi \not \in \Gamma_{N,r,\xi}$ for all $N,$ then there exists $C=C(\alpha, \delta, \ep,r)>0$ so that 
\[\mathbb P\lr{\forall \xi \in \Xi_N\quad \exists ! ~w\text{ inside }\Gamma_{N,\xi}~~\text{ s.t. }~~ \frac{d}{dw}p_{N,\Xi_N}(w)=0}\geq 1- C\cdot N^{-\delta}.\]
\end{Thm}
\begin{remark}
The conclusion of Theorem \ref{T:Main} is satisfied if $\Gamma_{N,\xi}$ is any contour satisfying (A) and conditions (i),(ii) from Remark \ref{R:GenCont}.
\end{remark}

There are other directions in which Theorem \ref{T:SCor} can be extended. We indicate some of them here.
\begin{enumerate}
\item The assumption that $\mu$ has a bounded density with respect to Haar measure ensures that typical spacings between the random zeros are like $N^{-1/2}.$ The intuitive electrostatic argument in \S \ref{S:Electrostatics} for why zeros and critical points are paired used only that zeros are weakly correlated and spaced more than $N^{-1}$ apart, however. This suggests that perhaps $\mu$ can be taken to be any measure satisfying a finite energy condition
\[\int_{S^2}\int_{S^2} \log \lr{d_{S^2}(z,w)}d\mu(w)d\mu(z)<\infty,\]
which rules out $\mu$ having an atom or being supported on a curve. 
\item The $1-C_\delta N^{-\delta}$ estimate from Theorem \ref{T:SCor} is sharp in the sense that with probability on the order of $N^{-1}$ there exists a $j$ so that $\abs{\xi_j-\xi}\leq N^{-1}.$ Such a zero will disrupt the local pairing of zeros and critical points. However, if one studies ensembles of polynomials for which zeros repel one another, then the $N^{-\delta}$ estimates can be improved. For instance, in \cite{H1,H2} the author studied such ensembles and showed that the probability of pairing is like $N^{-3/2}.$ Zeros will always repel for polynomials whose coefficients relative to a fixed basis are taken to be iid since the change of variables from coefficients to zeros involves a Vandermonde determinant. It should therefore be possible to generalize the results in this paper to the case when zeros are distributed like a Coulomb gas.
\item Theorem \ref{T:SCor} will not be true as stated for polynomials whose zeros tend to be $N^{-1}$ apart. Consider, for example, the Kac polynomials $p_N^{Kac}(z)=\sum_{j=0}^N a_j z^j$ with $a_j\sim N(0,1)_{\C}$ iid. The zeros are well-known to approximately equidistribute on the unit circle $S^1.$ Nonetheless it is clear from Figure 4 that most zeros are still paired to a unique critical point. A general family of such ensembles introduced by Shiffman-Zelditch in \cite{SZ} and further studied by Bloom in \cite{B} and Bloom-Shiffman in \cite{BS}.
\end{enumerate}

\begin{figure}[h]
\centering
\includegraphics[scale=.5]{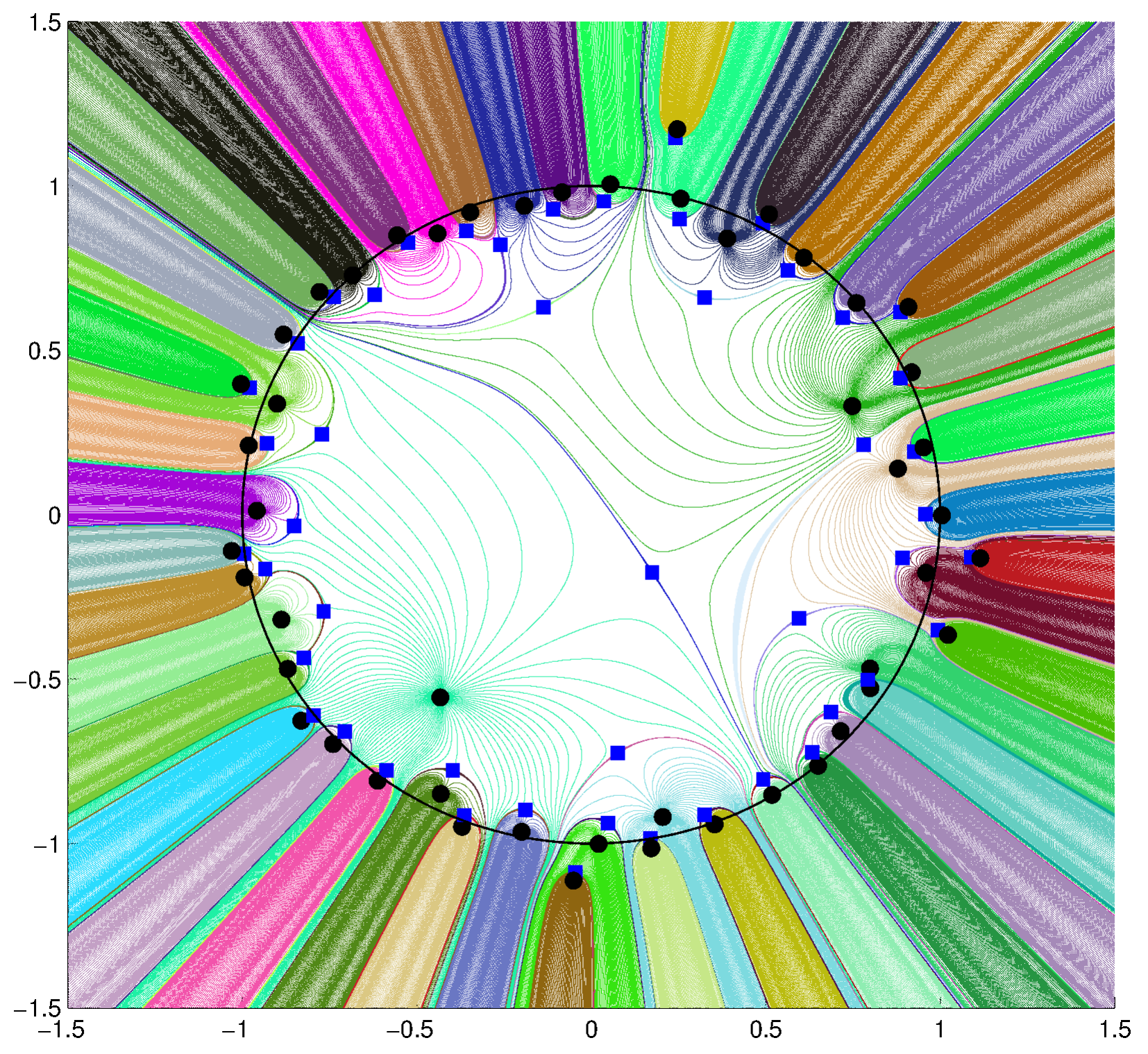}
\label{F:Scaled SU2 Crit}
\caption{Zeros (black disks) and critical points (blue squares) for $p_{50}^{Kac}.$ The colored lines are gradient flow lines for $\abs{p_{50}^{Kac}(z)}^2$. Lines of the same color terminate at the same point.}
\end{figure}

\section{Proof of Theorem \ref{T:SCor}}\label{S:CorPf} We prove Theorem \ref{T:SCor} when $\mu$ is the uniform measure on $S^2$ since the argument for general $\mu$ is identical and only involves carrying along various factors of $\phi_\mu(w)$. For the uniform measure, the average electric field at any fixed $w$ from one of the random zeros vanishes. To see this, we compute in polar coordinates around $w:$
\begin{align}
\notag  \phi_\mu(w)&=\int_{S^2} \frac{d\mu(\zeta)}{w-\zeta}=\int_{\C}\frac{1}{\zeta}\cdot \frac{id\zeta \wedge  d\overline{\zeta}}{2\pi\lr{1+\abs{\zeta}^2}^2}\\
&=\int_0^\infty \lr{1+r^2}^{-2}dr \cdot \int_0^{2\pi} e^{-i\theta}d\theta =0.\label{E:avgunif}
\end{align}
We therefore have
\[S_\mu=\set{\underline{0},\infty}.\] 
We work in the holomorphic coordinate $w$ centered at $\infty$ and fix $\xi\in S^2\backslash S_\mu.$ In our coordinates, $p_{N,\xi}$ is given by \eqref{E:patinfty} and $w_{\xi,N}$, which we shall henceforth abbreviate $w_\xi,$ is given by \eqref{E:critatinfty}. Write as in \S \ref{S:Electrostatics}
\[E_N(w)=-\frac{N}{w}+\sum_{j=0}^{N-1} \frac{1}{w-\xi_j},\]
and recall that critical points of $p_{N,\xi}$ are solutions to $E_N=0.$ The contour $\Gamma_N$ satisfies 
\begin{enumerate}
\item[(i)] There exists $c_1=c_1(r,\xi)>0$ so that 
\[\inf_{w\in \Gamma_N}\abs{-\frac{N}{w}+ \frac{1}{w-\xi}}\geq c_1 \cdot N.\]
\item[(ii)] There exists $c_2=c_2(r,\xi)>0$ so that
\[\sup_{w\in \Gamma_N}\abs{w-\xi}\leq c_2\cdot N^{-1},\]
\end{enumerate}
which are precisely the conditions from Remark \ref{R:GenCont}. Write
\begin{align*}
  \twiddle{E}_N(w)&:=E_N(w)-\E{E_N(w)}=\sum_{j=1}^{N-1} \frac{1}{w-\xi_j},
\end{align*}
and fix $\delta\in (0,1).$ We will show that there exists $\gamma=\gamma(c_1,c_2,\delta)>0$ and $C_3=C_3(c_1,c_2,\delta)$ so that 
\begin{equation}
\mathbb P\lr{\sup_{w\in \Gamma_N}\abs{\twiddle{E}_N(w)}\leq N^{1-\gamma}}\geq 1-C_3\cdot  N^{-\delta}.\label{E:Goal}
\end{equation}
The relation \eqref{E:Goal} and Rouch\'e's theorem would then imply that $p_{N,\xi}$ has a unique critical point inside $\Gamma_N$ with probability at least $1-C_3\cdot N^{-\delta},$ as desired. The proof of \eqref{E:Goal} is elementary but somewhat technical. Before giving the details we give a brief outline for the argument.
\subsection*{\bf Step 1.} Estimating the supremum of the random function $\twiddle{E}_N(w)$ restricted to $\Gamma_N$ is not simple to do directly. The basic reason is that $\twiddle{E}_N(w)$ fluctuates rather wildly (it does not even have a finite variance at a point). So instead we estimate separately the fluctuations of $\twiddle{E}_N(w)-\twiddle{E}_N(w_\xi)$ and of $\twiddle{E}_N(w_\xi).$ 
\subsection*{\bf Step 2.} To estimate $\twiddle{E}_N(w)-\twiddle{E}_N(w_\xi)$ we use that $\abs{w-w_\xi}\approx N^{-1}$ to throw away the contribution to $\twiddle{E}_N(w)-\twiddle{E}_N(w_\xi)$ coming from zeros that are far from $\xi$ (and hence from $w$ and $w_\xi$ as well). Specifically, by condition (ii), there exists $K_1=K_1(c_2,\delta)$ so that for all $N$
\begin{align*}
\sup_{w\in \Gamma_N}\abs{\twiddle{E}_N(w)-\twiddle{E}_N(w_\xi)}&=\sup_{w\in \Gamma_N}\abs{\sum_{\abs{\xi_j-\xi}>N^{-1/2+\delta/2}}\lr{\frac{1}{w-\xi_j}-\frac{1}{w_\xi-\xi_j}}}\\
&\leq \sup_{w\in \Gamma_N}\sum_{\abs{\xi_j-\xi}>N^{-1/2+\delta/2}}\frac{\abs{ w_\xi-w}}{\abs{w-\xi_j}\abs{w_\xi-\xi_j}}\\
&\leq K_1\cdot N^{1-\delta}.
\end{align*} 
\subsection*{\bf Step 3.} Writing
\[\twiddle{E}_N(w,\delta):=\sum_{\abs{\xi-\xi_j}\leq N^{-1/2+\delta/2}}\frac{1}{w-\xi_j},\] 
relation \eqref{E:Goal} now follows once we show that there exists $\gamma=\gamma(c_1,c_2,\delta)>0$ as well as $K_2=K_2(c_1,c_2, \delta)>0$ and $K_3=K_3(c_1,c_2,\delta)>0$ such that
\begin{equation}  
\mathbb P\lr{\sup_{w\in \Gamma_N}\abs{\twiddle{E}_N(w_\xi, \delta)-\twiddle{E}_N(w, \delta)}\geq N^{1-\gamma}}\leq K_2\cdot N^{-\delta}\label{E:Goal2}
\end{equation}
and 
\begin{equation}  
\mathbb P\lr{\abs{\twiddle{E}_N(w_\xi)}\geq N^{1-\gamma}}\leq K_3\cdot N^{-\delta}.\label{E:Goal3}
\end{equation}
\subsection*{\bf Step 4.} There are essentially two reasons that the events whose probabilities we seek to bound in \eqref{E:Goal2} and \eqref{E:Goal3} occur. First, if $\abs{\xi-\xi_j}\approx N^{-1}$ for some $j,$ then $\twiddle{E}_N(w),\twiddle{E}_N(w_\xi)$ will both be on the order of $N$ because of the single term involving $\xi_j.$ Second, if there are many more than $N^\delta$ zeros $\xi_j$ for which $\abs{\xi_j-\xi}\leq N^{-1/2+\delta/2},$ then each term in $\twiddle{E}_N(w_\xi,\delta)$ will be large enough that their sum could well be on the order of $N.$ However, both of these events themselves have small probability. To quantify this, we write for $\zeta \in S^2$
\[\mathcal N(\zeta, R):=\#\setst{j=1,\ldots, N\,\,}{\,\,d_{S^2}\lr{\xi_j,\zeta}\leq R},\]
where as before $d_{S^2}$ is the usual distance function on $S^2$ and prove the following lemma.  
\begin{Lem}\label{L:Number of Zeros}
Fix $\eta\in (0,\tfrac{1}{2}),\,\kappa >0.$ There exist $K=K(\eta)>0$ and $K'=K'(\kappa, \delta)>0$ so that
\begin{equation}
\mathbb P\lr{\mathcal N\lr{w_\xi, N^{-1+\eta}}\geq 1}\leq K \cdot N^{-1+2\eta}\label{E:SmallDisk}
\end{equation}
and
\[\mathbb P\lr{\mathcal N\lr{w_\xi, N^{-\frac{1}{2}+\frac{\delta}{2}}}\geq N^{\delta+\kappa}}\leq K'\cdot N^{-\delta-2\kappa}.\]
\end{Lem}
\subsection*{\bf Step 5.} Finally, note that the variance of $\twiddle{E}_N(w)$ is 
  \begin{equation}
\E{\abs{\twiddle{E}_N(w)}^2}=\sum_{j=1}^{N-1}\E{\frac{1}{\abs{w-\xi_j}^2}},\notag
\end{equation}
which is infinite. However, the conditional variance given $\mathcal N(w,N^{-1+\eta})=0$ is fairly small and allows us to get a good estimate on the tail probability in \eqref{E:Goal3}. This is the content of the following lemma.
\begin{Lem}\label{L:Fixed Efld}
Fix $\eta\in (0,\tfrac{1}{2}),$ and write $A$ for the event that
\[\sum_{j=1}^{N-1}\frac{1}{\abs{ w_\xi-\xi_j}^2}>N^{2-2\eta}.\]
There exists $K=K(\eta)$ such that
  \begin{equation}\label{E:Fixed Efld} 
\mathbb P\lr{A}\leq K \cdot N^{-1+2\eta}\log N.
  \end{equation}
\end{Lem}
We now turn to the details. 
\begin{proof}[Proof of Lemma \ref{L:Number of Zeros}]
The estimate \eqref{E:SmallDisk} is true since there is a constant $K=K(\eta)>0$ so that for every $w\in S^2,$ 
\[\mathbb P\lr{\mathcal N\lr{w, N^{-1+\eta}}\geq 1}\leq \E{\mathcal N\lr{w, N^{-1+\eta}}}=K_1 \cdot N^{-1+2\eta}.\]
Next, there is a constant $K'$ so that for any $w\in S^2,$ the random variable $\mathcal N\lr{w, N^{-\frac{1}{2}+\frac{\delta}{2}}}$ has a binomial distribution with number of trials $N-1$ and success probability $p$ not exceeding $K\cdot N^{-1+\delta}.$ Therefore, 
\[Var\left[\mathcal N\lr{w, N^{-\frac{1}{2}+\frac{\delta}{2}}}\right]\leq K_2\cdot N^{\delta}.\]
Hence, by Chebyshev's inequality, 
\[\mathbb P\lr{\mathcal N\lr{w, N^{-\frac{1}{2}+\frac{\delta}{2}}}\geq N^{\delta+\kappa}}\leq K_2\cdot N^{-\delta -2\kappa},\]
as claimed. 
\end{proof}

\begin{proof}[Proof of Lemma \ref{L:Fixed Efld}]
Define the event
\[ B=\set{\mathcal N(w_\xi, N^{-1+\eta})\geq 1}.\]
Computing in polar coordinates around $w_\xi,$ we find that there exists $C>0$ satisfying
\begin{align}
\E{\sum_{j=1}^{N-1}\frac{1}{\abs{w_\xi-\xi_j}^2}~\bigg|~B}&\leq C\cdot N \int_{N^{-1+\eta}}^{\infty}\frac{dr}{r\lr{1+r^2}^2}\leq C\lr{1-\eta}\cdot N\log N.\label{E:RHS}
\end{align}
By Lemma \ref{L:Number of Zeros}, there is a constant $c=c(\eta)$ so that
\begin{align}
  \mathbb P\lr{B}\geq 1- c\cdot N^{-1+2\eta}.\label{E:Best}
\end{align}
Therefore, 
\begin{align*}
\abs{\mathbb P(A)-\mathbb P(A\,|\, B)}\leq c \cdot N^{-1+2\eta}.
\end{align*}
Combining Markov's inequality with \eqref{E:RHS}, we find that there exists $C=C(\eta)$ for which 
\[\mathbb P(A\,|\, B)\leq 2C(1-\eta)N^{-1+2\eta}\log N,\]
completing the proof. 
\end{proof}
\noindent We are ready to show \eqref{E:Goal2}. We estimate the modulus of 
\[\twiddle{E}_N(w_\xi, \delta)-\twiddle{E}_N(w, \delta)=\sum_{\abs{w_\xi-\xi_j}\leq N^{-1/2+\delta/2}}\frac{w_\xi-w}{\lr{w-\xi_j}\lr{w-\xi_j}}\]
by using the constant $c_2$ from assumption (ii), to find that for all $w\in \Gamma_N$
\begin{align*}
\abs{  \twiddle{E}_N(w, \delta)-\twiddle{E}_N(w_\xi,\delta)}&\leq c_2\cdot N^{-1} \abs{\sum_{\abs{\xi-\xi_j}\leq N^{-1/2+\delta/2}}\frac{1}{\lr{w_\xi-\xi_j}\lr{w-\xi_j}}}.
\end{align*}
Adding and subtracting $\lr{w_\xi-\xi_j}^{-2}$ inside the absolute values, we find that the right hand side of the previous line is bounded above by 
\begin{align}
c_2^2\cdot N^{-2} \abs{\sum_{\abs{\xi-\xi_j}\leq N^{-1/2+\delta/2}}\frac{1}{\lr{w_\xi-\xi_j}^2\lr{w-\xi_j}}}+ c_2\cdot N^{-1}\abs{ \sum_{\abs{\xi-\xi_j}\leq N^{-1/2+\delta/2}}\frac{1}{\lr{w_\xi-\xi_j}^2}}.\label{E:Res}
\end{align}
Continuing in this way, for every $l\geq 1,$ we may write
\begin{align}
\label{E:RHS2}  \abs{ \twiddle{E}_N(w, \delta)-\twiddle{E}_N(w_\xi,\delta)}&\leq c_2^l N^{-l}\cdot \abs{ \sum_{\abs{\xi-\xi_j}\leq N^{-1/2+\delta/2}}\frac{1}{\lr{w_\xi-\xi_j}^l\lr{w-\xi_j}}}\\
\label{E:RHS3}  &+\sum_{k=1}^{l-1} c_2^k N^{-k}\abs{\sum_{\abs{\xi-\xi_j}\leq      N^{-1/2+\delta/2}}\frac{1}{\lr{w_\xi-\xi_j}^{k+1}} }.
\end{align}
The key point is that $w$ appears only in \eqref{E:RHS2}, while \eqref{E:RHS3} involves only $w_\xi.$ Choose $l$ large enough so that
\begin{equation}
\delta<\frac{l+1}{l+5}.\label{E:ldef}
\end{equation}
Lemma \ref{L:Number of Zeros} shows that there exists $C=C(\delta)>0$ so that with probability at least $1-C\cdot N^{-\delta},$ 
\[\mathcal N\lr{w_\xi, N^{-\frac{1}{2}-\frac{\delta}{2}}}=0\qquad \text{and}\qquad N\lr{w_\xi, N^{-\frac{1}{2}+\frac{\delta}{2}}}\leq N^{2\delta}.\]
Hence, using assumptions (i) and (ii), there exists $C'=C'(c_1,c_2,\delta)$ and $C''=C''(\delta)>0$ so that with probability $1-C''\cdot N^{-\delta}$
\[N^{-l}  \sum_{\abs{w_\xi-\xi_j}\leq N^{-1/2+\delta/2}}\abs{\frac{1}{\lr{w_\xi-\xi_j}^l\lr{w-\xi_j}}}\leq C'\cdot N^{1-(l+1-2\delta - (l+1)\lr{\frac{1+\delta}{2}})}.\]
Using \eqref{E:ldef}, we have $\gamma := l+1-2\delta - (l+1)\lr{\frac{1+\delta}{2}}>0,$ which shows that the right hand side of \eqref{E:RHS2} is bounded by $N^{1-\gamma}$ with probability as least $1-C''\cdot N^{-\delta}$. To bound \eqref{E:Res}, we apply Lemma \ref{L:Fixed Efld} for some fixed $\eta\in \lr{\frac{1-\delta}{2},\frac{1}{2}}$ to find that there exists $C=C(\delta)$ and $C'=C'(c_1,c_2,\delta)$ so that
\begin{align*}
  \sum_{k=1}^{l-1}c_2^k N^{-k} \abs{ \sum_{\abs{w_\xi-\xi_j}\leq N^{-1/2+\delta/2}}\frac{1}{\lr{w_\xi-\xi_j}^{k+1}}}&\leq  \sum_{k=1}^{l=1}c_2^k N^{-k} \lr{\sum_{\abs{w_\xi-\xi_j}\leq N^{-1/2+\delta/2}}\frac{1}{\abs{ w_\xi-\xi_j}^2}}^{\frac{k+1}{2}}\\
&\leq C'\cdot \sum_{k=1}^{l-1} N^{-k + (1-\eta)(k+1)}\\
&\leq C'(l-1)\cdot N^{\delta}
\end{align*}
with probability at least $1-C\cdot N^{-\delta},$ proving that \eqref{E:Goal2} holds. Finally, we show \eqref{E:Goal3}.
Set $\eta=\frac{1}{2}\lr{1-\delta}$ and recall the event $B$ from Lemma \ref{L:Fixed Efld}.  Observe that
\[\E{\abs{\twiddle{E}_N(w_\xi)}^2\, \big| \, B}=\E{\sum_{j,k=1}^{N-1}\lr{\frac{1}{w_\xi-\xi_j}}\cdot \lr{\overline{\frac{1}{w_\xi-\xi_k}}}\, \big|\,  B}=\E{\sum_{j=1}^{N-1} \frac{1}{\abs{w_\xi-\xi_j}^2}\, \big| \, B}.\] Using \eqref{E:Best}, Markov's inequality and \eqref{E:RHS}, we have that for all $\gamma\in \lr{0,\frac{1-\delta}{2}}$ there exists $C=C(\delta),C'=C'(\gamma),C''=C''(\gamma, \delta)$ so that
\begin{align*}
  \mathbb P\lr{\abs{\twiddle{E}_N(w_\xi)}>N^{1-\gamma}}&\leq \mathbb P\lr{\abs{\twiddle{E}_N(w_\xi)}>N^{1-\gamma}\,\big|\, B}+C\cdot N^{-\delta}\\
&\leq N^{-2+2\gamma} \cdot \E{\sum_{j=1}^{N-1}\frac{1}{\abs{ w_\xi-\xi_j}^2}\, \bigg| \, B}+C\cdot N^{-\delta}\\
&=C' \cdot N^{-1+2\gamma}\log N + C \cdot N^{-\delta}\\
&\leq C'' \cdot N^{-\delta}.
\end{align*}
This completes the proof of Theorem \ref{T:SCor}.
\section{Proof of Theorem \ref{T:Main}}
Fix $\alpha,\ep,\Xi_N$ as in the statement of Theorem \ref{T:Main} as well as $\delta \in (0,\alpha)$. Fix $\und{\xi}\in \Xi_N$ and write $Z_N=p_{N,\Xi_N}^{-1}(0).$ We have 
\begin{align}
\notag  \twiddle{E}_{N,\und{\xi}}(w)&=\sum_{\xi\in Z_N\backslash \set{\und{\xi}}}\frac{1}{w-\xi}\\
&=\sum_{\xi\in \Xi_N\backslash \set{\und{\xi}}}\frac{1}{w-\xi}+ \sum_{\xi\in Z_N\backslash \Xi_N}\frac{1}{w-\xi}.\label{E:TwoTerms}
\end{align}
By the same argument as in the proof of Theorem \ref{T:SCor}, there exists $\gamma>0$ and $C=C(\delta)$ so that 
\begin{equation}
\mathbb P\lr{\sup_{w\in \Gamma_{N,\und{\xi}}}\abs{\sum_{\xi\in Z_N\backslash \Xi_N}\frac{1}{w-\xi}}>N^{1-\gamma}}\leq C\cdot N^{-\delta}.\label{E:easyterm}
\end{equation}
Note that $C$ is independent of $\und{\xi}.$ To estimate the first term in \eqref{E:TwoTerms}, note that the well-spacing assumption (A) implies that 
\begin{equation*}
  \#\,\,\Xi_N\cap \setst{w\in \C}{jN^{-1/2}<\abs{w-\und{\xi}}\leq (j+1)N^{-1/2}} \leq \frac{2j+1}{N}\cdot N^{1-\ep}=(2j+1)N^{-\ep}.
\end{equation*}
Therefore, 
\begin{align*}
\sup_{w\in \Gamma_{N,\und{\xi}}}  \sum_{\xi\in \Xi_N\backslash \set{\und{\xi}}} \frac{1}{\abs{w-\xi}} &\leq \sum_{\substack{\xi \in \Xi_N\backslash \set{\und{\xi}}\\ \abs{\xi-\und{\xi}}\leq N^{\ep/2}}}\frac{1}{\abs{w-\xi}}\,\, +\sum_{\substack{\xi \in \Xi_N\backslash \set{\und{\xi}}\\ \abs{\xi-\und{\xi}}>N^{\ep/2}}}\frac{1}{\abs{w-\xi}}\\
&\leq \sum_{j=1}^{N^{\ep/2+1/2}} (2j+1)N^{-\ep}\cdot \frac{N^{1/2}}{j}\,\, +N^{1-\ep/2}\\
&\leq 4\cdot N^{1-\ep/2}.
\end{align*}
A simple union bound now shows that there exists $C=C(\alpha,\delta,\ep)>0$ so that
\begin{equation*}
\mathbb P\lr{\sup_{\und{\xi}\in \Xi_N}\sup_{w\in \Gamma_{N,\und{\xi}}}\abs{\twiddle{E}_N(w)}> N^{1-\gamma}}\leq C\cdot N^{\alpha-\delta},
\end{equation*}
completing the proof.

\end{document}